\documentclass[a4paper,11pt]{article}
\usepackage[]{graphicx}
\usepackage[]{amsmath,amssymb}
\usepackage[]{amsthm,enumerate}
\usepackage{verbatim,bm}
\usepackage{color}
\usepackage{a4wide}

\newtheorem{theorem}{Theorem}[section]
\newtheorem{lemma}[theorem]{Lemma}
\newtheorem{corollary}[theorem]{Corollary}
\newtheorem{proposition}[theorem]{Proposition}

\theoremstyle{definition}

\newcommand{\defn}[1]{{\em #1}}
\theoremstyle{remark}

\title{Strongly regular decompositions and symmetric association schemes of a power of two}
\date{\today}

\author{
 Hadi Kharaghani\thanks{Department of Mathematics and Computer Science, University of Lethbridge,
Lethbridge, Alberta, T1K 3M4, Canada.  \texttt{kharaghani@uleth.ca}} 
\and  
 Sara Sasani\thanks{Department of Mathematics and Computer Science, University of Lethbridge,
Lethbridge, Alberta, T1K 3M4, Canada.  \texttt{sasani@uleth.ca}} 
\and
 Sho Suda\thanks{Department of Mathematics Education,  Aichi University of Education, 1 Hirosawa, Igaya-cho, Kariya, Aichi 448-8542, Japan. \texttt{suda@auecc.aichi-edu.ac.jp}}}

\begin{document}
\maketitle
\abstract{
For any positive integer $m$, the complete graph on $2^{2m}(2^m+2)$ vertices is decomposed into $2^m+1$ commuting strongly regular graphs, which give rise to a symmetric association scheme of class $2^{m+2}-2$.
Furthermore, the eigenmatrices of the symmetric association schemes are determined explicitly. 
As an application, the eigenmatrix of the commutative strongly regular decomposition obtained from the strongly regular graphs is derived.   
}

\section{Introduction}
A \defn{strongly regular graph with parameters $(v,k,\lambda,\mu)$} is a regular graph with $v$ vertices of degree $k$ such that every two adjacent vertices have exactly $\lambda$ common neighbors and  every two non-adjacent vertices have exactly $\mu$ common neighbors. 
A \defn{strongly regular decomposition} is a decomposition of the edge set of the complete graph with vertex set $V$ into strongly regular graphs with vertex set $V$. 
A strongly regular decomposition is \defn{commutative} if the adjacency matrices of strongly regular graphs are commutative. 
The concept of strongly regular decompositions was introduced by van Dam \cite{D} in order to study more general situation of amorphous association schemes.

In this paper, we show that for any positive integer $m$, there is a commutative strongly regular decomposition of the complete graph of $2^{2m}(2^m+2)$ vertices, into $2^m$ strongly regular graphs with parameters $(2^{2m}(2^m+2),2^{2m}+2^m,2^m,2^m)$ and $2^m+2$ cliques of size $2^{2m}$. 
Note that  $2^m+2$ cliques of size $2^{2m}$ is a strongly regular graph with parameters $(2^{2m}(2^m+2),2^{2m}-1,2^{2m}-2,0)$. 
In fact, the constructed strongly regular graphs with parameters $(2^{2m}(2^m+2),2^{2m}+2^m,2^m,2^m)$ are symmetric $(2^{2m}(2^m+2),2^{2m}+2^m,2^m)$-designs with symmetric incidence matrices with very large number of symmetries.  
Our construction method is based on the generalized Hadamard matrices obtained from finite fields of characteristic two and symmetric Latin squares with constant diagonal of even order. 

One might wonder if the decomposition yields a symmetric association scheme, but unfortunately this does not hold. 
We had to further decompose the edge sets of the strongly regular graphs to obtain a symmetric association scheme, and determine the eigenamtrices of the symmetric association scheme explicitly. 
As a corollary, we obtain the eigenmatrix of the commutative strongly regular decomposition.

%%%%%%%%%%%%%%%%%%%%%%%%%%%%%%%%%%%%%%%%%%%%%%%%%%%%%%%%%%%%%%%%%%%%%%%%%%%%%%%%%%%%%%%%%%%%%%
\section{Preliminaries}
Let $n$ be a positive integer. 
Let $V$ be a finite set of size $v$ and $R_i$ ($i\in\{0,1,\ldots,n\}$) be a non-empty subset of $V\times V$. 
The \emph{adjacency matrix} $A_i$ of the graph with vertex set $V$ and edge set $R_i$ is a $v\times v$ $(0,1)$-matrix with rows and columns indexed by the elements of $V$ such that $(A_i)_{xy}=1$ if $(x,y)\in R_i$ and $(A_i)_{xy}=0$ otherwise. 
The pair $(V,\{R_i\}_{i=0}^n)$ is said to be a \defn{commutative decomposition} of the complete graph if the following hold:
\begin{enumerate}
\item $A_0=I_v$, the identity matrix of order $v$.
\item $\sum_{i=0}^n A_i = J_v$, the all-ones matrix of order $v$.
\item $A_i$ is symmetric for $i\in\{1,\ldots,n\}$.
\item For any $i,j$, $A_i A_j=A_j A_i$. 
\end{enumerate}
We also refer to the set of non-zero $v\times v$ $(0,1)$-matrices satisfying (i)-(iv) as a commutative decomposition. 
Note that the corresponding graph of each $A_i$ is regular, because $A_i$ and $J_v=\sum_{i=0}^n A_i$ commute. 
Let $k_i$ denote the valency of the corresponding graph of $A_i$. 
A commutative decomposition $\{A_0,A_1,\ldots,A_n\}$ is a \defn{symmetric association scheme of class $n$} if there exist non-negative integers $p_{i,j}^k$ such that $A_i A_j=\sum_{k=0}^n p_{i,j}^k A_k$.
The non-negative integers $p_{i,j}^k$ are said to be the \defn{intersection numbers}. 
A commutative decomposition is a \defn{strongly regular decomposition} if the corresponding graph of each $A_i$ is strongly regular. 
Note that a $v\times v$ $(0,1)$-matrix $A$ is the adjacency matrix of a strongly regular graph if and only if $\{I_v,A,J_v-A-I_v\}$ is a symmetric association scheme of class $2$. 

We define the \defn{eigenmatrix} for commutative decompositions. 
Since the adjacency matrices are symmetric and commuting, there are maximal common eigenspaces $V_0,V_1,\ldots,V_t$ of $A_i$'s. 
Let $E_j$ be the orthogonal projection $\mathbb{R}^v$ onto $V_j$ for $j\in\{0,1,\ldots,t\}$.   
Note that the corresponding graph of each $A_i$ is regular and one of them is connected, so we may assume that $V_0$ is spanned by the all-ones vector and thus $E_0=(1/v)J_v$. 
The matrices $E_0,E_1,\ldots,E_t$ satisfy $E_iE_j=\delta_{ij}E_i$ for any $i,j$ where $\delta_{ij}$ is the Kronecker delta. 
Then we can express $A_j$ as a linear combination of $E_i$:
\begin{align*}
A_j=\sum_{i=0}^t p_{ij}E_i.
\end{align*}
We call a $(t+1)\times (n+1)$ matrix $P=(p_{ij})_{\substack{0\leq i\leq t\\0\leq j \leq n}}$ the \defn{eigenmatrix} of the commutative decomposition.
The following lemma was proved in \cite{D}. 
\begin{lemma}
Let $(V,\{R_i\}_{i=0}^n)$ be a commutative decomposition with orthogonal projections $E_j$, $j\in\{0,1,\ldots,t\}$. 
Then $t\geq n$ with equality if and only if $(V,\{R_i\}_{i=0}^n)$ is a symmetric association scheme.  
\end{lemma}
For the case of symmetric association schemes, we have $t=n$ so that the eigenmatrix $P$ becomes a square matrix of order $n+1$. 
The vector space $\mathcal{A}$ spanned by the adjacency matrices over $\mathbb{R}$ is closed under the matrix multiplication and is said to be the \defn{Bose-Mesner algebra}.
Then each $E_j$ is an element in $\mathcal{A}$, so it is written as a linear combination of $A_i$ as follows: for some real numbers $q_{ij}$,   
\begin{align*}
E_j=\frac{1}{v}\sum_{i=0}^n q_{ij}A_i.
\end{align*}
The matrices $E_0,E_1,\ldots,E_n$ are said to be the \defn{primitive idempotents} of the symmetric association scheme. 
The matrix $P$ is also said to be the \defn{first} eigenmatrix and the matrix $Q=(q_{ij})_{i,j=0}^{n}$ is said to be the \defn{second eigenmatrix} of the symmetric association scheme. 
The valencies $k_j$ ($j\in\{0,1,\ldots,n\}$) satisfy $k_j=p_{0j}$, and the \defn{multiplicities} $m_j$ ($j\in\{0,1,\ldots,n\}$) are defined by $m_j=\mathrm{rank}(E_j)$ and satisfy $m_j=q_{0j}$. 
The first and second eigenmatrices are related as the following. 
\begin{lemma}\label{lem:eigen}\upshape{\cite[Theorem~3.5(i)]{BI}}
Let $(V,\{R_i\}_{i=0}^n)$ be a symmetric association scheme with first and second eigenmatrices $P$ and $Q$. 
Let $\Delta_k$ and $\Delta_m$ be $(n+1)\times (n+1)$ diagonal matrices with diagonal entries $k_0,k_1,\ldots,k_n$ and $m_0,m_1,\ldots,m_n$, respectively.  
Then it holds that $\Delta_m P=Q^\top \Delta_k$ where $Q^\top$ is the transpose of $Q$. 
\end{lemma}
\begin{proof}
Calculate the trace of $A_iE_j$ in two ways.
\end{proof}
Finally we define the Krein numbers of the symmetric association scheme. 
Since the Bose-Mesner algebra has a basis $\{A_0,A_1,\ldots,A_n\}$ consisting of disjoint $(0,1)$-matrices, $\mathcal{A}$ is closed under the entrywise product denoted $\circ$.  
Then for any $i,j,k\in\{0,1,\ldots,n\}$, there exist real numbers $q_{i,j}^k$ such that $E_i\circ E_j=\frac{1}{v}\sum_{k=0}^n q_{i,j}^k E_k$.  
The real numbers $q_{i,j}^k$ are said to be the \defn{Krein numbers}, and it is known that the Krein numbers are non-negative real numbers \cite[Theorem~3.8]{BI}.

%%%%%%%%%%%%%%%%%%%%%%%%%%%%%%%%%%%%%%%%%%%%%%%%%%%%%%%%%%%%%%%%%%%%%%%%%%%%%%%%%%%%%%%%%%%%%%
\section{Association schemes of a power of two}
From now let $q=2^m$ be a power of two. 
We denote by $\mathbb{F}_q$ the finite field of $q$ elements.   
Let $H_q$ be the multiplicative table of $\mathbb{F}_q$, i.e., $H_q$ is a $q\times q$ matrix with rows and columns indexed by the elements of $\mathbb{F}_q$ with $(\alpha,\beta)$-entry equal to $\alpha \cdot \beta$. 
Then the matrix $H_q$ is a \defn{generalized Hadamard matrix with parameters $(q,1)$} over the additive group of $\mathbb{F}_q$. 
Letting $G$ be an additively written finite abelian  group of order $g$, 
a square matrix $H=(h_{ij})_{i,j=1}^{g\lambda}$ of order $g\lambda$ with entries from $G$ is called a {\it generalized Hadamard matrix with the parameters $(g,\lambda)$} over $G$ 
if for all distinct $i,k\in\{1,2,\ldots,g\lambda\}$, the multiset $\{h_{ij}-h_{kj}: 1\leq j\leq g\lambda\}$ contains exactly $\lambda$ times of each element of $G$. 

Let $\phi$ be a permutation representation of the additive group of $\mathbb{F}_q$ defined as follows.  
Since $q=2^m$, we view the additive group of $\mathbb{F}_q$ as $\mathbb{F}_2^m$. 
Define $R=J_2-I_2$, and a group homomorphism $\phi:\mathbb{F}_{q}\rightarrow GL_{q}(\mathbb{R})$ as $\phi((x_i)_{i=1}^m)= \otimes_{i=1}^m R^{x_i}$. 

From the generalized Hadamard matrix $H_q$ and the permutation representation $\phi$, we construct $q^2$ auxiliary matrices; 
for each $\alpha,\alpha'\in \mathbb{F}_q$, define a $q^2\times q^2$ $(0,1)$-matrix $C_{\alpha,\alpha'}$ to be 
\begin{align*}
C_{\alpha,\alpha'}=(\phi(\alpha(-\beta+\beta')+\alpha'))_{\beta,\beta'\in\mathbb{F}_q}. 
\end{align*}
Further, let $x,y$ be indeterminates, we define $C_{x,\alpha},C_{y,\alpha}$ by $C_{x,\alpha}=O_{q^2}$ and $C_{y,\alpha}=\phi(\alpha)\otimes J_q$ for $\alpha\in\mathbb{F}_q$, where $O_{q^2}$ denotes the zero matrix of order $q^2$. 

It is known that a symmetric Latin square of order $v$ with constant diagonal exists for any positive even integer $v$, see \cite{K}. 
Let $L=(L(a,a'))_{a,a'\in S}$ be a symmetric Latin square of order $q+2$ on the symbol set $S=\mathbb{F}_q\cup\{x,y\}$ with  constant diagonal $x$. 
Write $L$ as $L=\sum_{a\in S}a\cdot P_a$, where $P_a$ is a symmetric permutation matrix of order $q+2$. % and $P_x=I_{q+2}$. 
Note that $P_x=I_{q+2}$. 

From the $(0,1)$-matrices $C_{\alpha,\alpha'}$'s and the Latin square $L$, we construct symmetric designs as follows. 
For $\alpha \in \mathbb{F}_q$, we define a $(q+2)q^2\times (q+2)q^2$ $(0,1)$-matrix $N_{\alpha}$ to be 
\begin{align*}
N_\alpha=(C_{L(a,a'),\alpha})_{a,a'\in S}=\sum_{a\in \mathbb{F}_q\cup\{y\}} P_a\otimes C_{a,\alpha}. 
\end{align*}
In order to show that each $N_{\alpha}$ is a symmetric design and study more properties, we prepare a lemma on $C_{\alpha,\alpha'}$ and $P_a$.  
\begin{lemma}\label{lem:1}
\begin{enumerate}
\item For $\alpha\in\mathbb{F}_q$, $\sum_{a\in\mathbb{F}_q\cup\{y\}}C_{a,\alpha}=q I_q\otimes \phi(\alpha)+(J_q+\phi(\alpha)-I_q)\otimes J_q$. 
\item For $a\in\mathbb{F}_q\cup\{y\}$ and $\alpha,\alpha'\in\mathbb{F}_q$,
$C_{a,\alpha}C_{a,\alpha'}=q C_{a,\alpha+\alpha'}$.
\item For distinct $a,a'\in\mathbb{F}_q\cup\{y\}$ and $\alpha,\alpha'\in\mathbb{F}_q$, $C_{a,\alpha}C_{a',\alpha'}=J_{q^2}$.
\item For $\alpha,\alpha',\alpha''\in\mathbb{F}_q$, $(I_q\otimes \phi(\alpha''))C_{\alpha,\alpha'}=C_{\alpha,\alpha'+\alpha''}$. 
\item For $\alpha,\alpha'\in\mathbb{F}_q$, $(I_q\otimes \phi(\alpha))C_{y,\alpha'}=C_{y,\alpha'}$. 
\item $\sum_{a,b\in \mathbb{F}_q\cup\{y\},a\neq b} P_{a}P_{b}=q(J_{q+2}-I_{q+2})$. 
\end{enumerate}
\end{lemma}
\begin{proof}
(i): For $\alpha,\beta,\beta'\in\mathbb{F}_q$, the $(\beta,\beta')$-entry of $\sum_{\gamma\in\mathbb{F}_q}C_{\gamma,\alpha}$ is 
\begin{align*}
\sum_{\gamma\in\mathbb{F}_q}\phi(\gamma(-\beta+\beta')+\alpha)&=\begin{cases}\sum_{\gamma\in\mathbb{F}_q}\phi(\alpha) & \text{ if } \beta=\beta' \\ \sum_{\gamma'\in\mathbb{F}_q}\phi(\gamma'+\alpha) & \text{ if } \beta\neq \beta' \end{cases}\\
&=\begin{cases}q\phi(\alpha) & \text{ if } \beta=\beta', \\ J_q & \text{ if } \beta\neq \beta', \end{cases}
\end{align*}
which yields $\sum_{\gamma\in\mathbb{F}_q}C_{\gamma,\alpha}=q I_q\otimes \phi(\alpha)+(J_q-I_q)\otimes J_q$. 
Therefore 
\begin{align*}
\sum_{a\in\mathbb{F}_q\cup\{y\}}C_{a,\alpha}=\sum_{\gamma\in\mathbb{F}_q}C_{\gamma,\alpha}+C_{y,\alpha}
=q I_q\otimes \phi(\alpha)+(J_q+\phi(\alpha)-I_q)\otimes J_q.
\end{align*} 

(ii): For $a=y$, $C_{y,\alpha}C_{y,\alpha'}=(\phi(\alpha)\otimes J_q)(\phi(\alpha')\otimes J_q)=q\phi(\alpha+\alpha')\otimes J_q=qC_{y,\alpha+\alpha'}$. 
For $a,\beta,\beta'\in\mathbb{F}_q$, the $(\beta,\beta')$-entry of $C_{a,\alpha}C_{a,\alpha'}$ is 
\begin{align*}
\sum_{\gamma\in\mathbb{F}_q}\phi(a(-\beta+\gamma)+\alpha)\phi(a(-\gamma+\beta')+\alpha')
&=\sum_{\gamma\in\mathbb{F}_q}\phi(a(-\beta+\beta')+\alpha+\alpha')\\
&=q\phi(a(-\beta+\beta')+\alpha+\alpha'). 
\end{align*}
Thus we have $C_{a,\alpha}C_{a,\alpha'}=q C_{a,\alpha+\alpha'}$. 

(iii): The case of $a\in\mathbb{F}_q$ and $a'=y$ follows from the fact that $C_{a,\alpha}$ is a block matrix whose $q\times q$ sub-block is a permutation matrix. The case of $a,a'\in\mathbb{F}_q,a\neq a'$ follows from a similar calculation to (ii) with the fact that $ \{(a-a')\gamma\mid \gamma\in\mathbb{F}_q\}=\mathbb{F}_q$. 

(iv) and (v) are routine, and (vi) follows from the equations below. Recall that $S=\mathbb{F}_q\cup\{x,y\}$. 
\begin{align*}
\sum_{a,b\in \mathbb{F}_q\cup\{y\},a\neq b} P_{a}P_{b}&=\sum_{a\in\mathbb{F}_q\cup\{y\}} P_{a}(\sum_{b\in S\setminus\{x,a\}}P_{b})\\
&=\sum_{a\in\mathbb{F}_q\cup\{y\}} P_{a}(J_{q+2}-I_{q+2}-P_{a})\\
&=\sum_{a\in\mathbb{F}_q\cup\{y\}} (J_{q+2}-P_{a}-I_{q+2})\\
&=(q+1)(J_{q+2}-I_{q+2})-\sum_{a\in\mathbb{F}_q\cup\{y\}}P_{a}\\
&=q(J_{q+2}-I_{q+2}). \qedhere
\end{align*}
\end{proof}

Now we are ready to prove results for $N_{\alpha}$'s. 
Note that the result for $N_0$ being a symmetric $((q+2)q^2,q^2+q,q)$-design is well-known, see for example \cite[Exercise~5.7]{S}. 
\begin{theorem}\label{thm:1}
\begin{enumerate}
\item For any $\alpha\in\mathbb{F}_q$, $N_{\alpha}$ is symmetric. 
%\item For any $i\in\{1,\ldots,q\}$, $N_i^2=q^2 I_{(q+2)q^2}+q J_{(q+2)q^2}$. 
\item For any $\alpha,\beta\in\mathbb{F}_q$, 
\begin{align*}
N_{\alpha} N_{\beta}%&= q I_{q+2}\otimes\sum_{b=1}^{q+1}C_{b,m} +q(J_{q+2}-I_{q+2})\otimes J_{q^2}\\
= q^2 I_{q(q+2)}\otimes \phi(\alpha+\beta)+q I_{q+2}\otimes \phi(\alpha+\beta)\otimes J_q+q (J_{q(q+2)}-I_{q(q+2)})\otimes J_{q}. 
\end{align*}
In particular, $N_{\alpha}^2=q^2 I_{(q+2)q^2}+q J_{(q+2)q^2}$, that is, $N_\alpha$ is the incidence matrix of a symmetric $((q+2)q^2,q^2+q,q)$-design.
\end{enumerate}
\end{theorem}
\begin{proof}
(i): It follows from the properties that the matrices $P_a$ and $C_{a,\alpha}$ are symmetric for $a\in\mathbb{F}_q\cup\{y\}$ and $\alpha\in\mathbb{F}_q$. 

(ii): We use Lemma~\ref{lem:1} to obtain:  
\begin{align*}
N_{\alpha} N_{\beta} &%=(\sum_{a\in\mathbb{F}_q\cup\{y\}} P_a\otimes C_{a,\alpha})(\sum_{b\in \mathbb{F}_q\cup\{y\}} P_{b}\otimes C_{b,\beta})
=\sum_{a,b\in \mathbb{F}_q\cup\{y\}} P_{a}P_{b}\otimes C_{a,\alpha}C_{b,\beta}\\
&=\sum_{a\in\mathbb{F}_q\cup\{y\}} P_{a}^2\otimes C_{a,\alpha}C_{a,\beta}+\sum_{a,b\in\mathbb{F}_q\cup\{y\},a\neq b} P_{a}P_{b}\otimes C_{a,\alpha}C_{b,\beta}\\
&=\sum_{a\in\mathbb{F}_q\cup\{y\}} I_{q+2}\otimes q C_{a,\alpha+\beta}+\sum_{a,b\in\mathbb{F}_q\cup\{y\},a\neq b} P_{a}P_{b}\otimes J_{q^2}\\
&=q I_{q+2}\otimes (qI_q\otimes \phi(\alpha+\beta)+(J_q+\phi(\alpha+\beta)-I_q)\otimes J_q)+q(J_{q+2}-I_{q+2})\otimes J_{q^2}\\
&=q^2 I_{q(q+2)}\otimes \phi(\alpha+\beta)+qI_{q+2}\otimes \phi(\alpha+\beta)\otimes J_q+q (J_{q(q+2)}-I_{q(q+2)})\otimes J_{q}. 
\end{align*}
The case of $\alpha=\beta$ follows from $\phi(0)=I_q$. 
\end{proof}

Now we obtain our main result in this section. 
\begin{theorem}
The set of matrices $\{N_{\alpha},I_{q+2}\otimes (J_{q^2}-I_{q^2}),I_{(q+2)q^2}\mid \alpha\in \mathbb{F}_q \}$ is a commutative strongly regular decomposition of the complete graph on $(q+2)q^2$ vertices. 
\end{theorem}
\begin{proof}
It is easy to see that $\sum_{\alpha\in\mathbb{F}_q} N_{\alpha}+I_{q+2}\otimes (J_{q^2}-I_{q^2})=J_{(q+2)q^2}-I_{(q+2)q^2}$.  
By Theorem~\ref{thm:1} (i), (ii) for $\alpha=\beta$, the matrices $N_{\alpha}$, $\alpha\in\mathbb{F}_q$, are the adjacency matrices of strongly regular graphs with parameters $((q+2)q^2,q^2+q,q,q)$.
And $I_{q+2}\otimes (J_{q^2}-I_{q^2})$ is the adjacency matrix of a strongly regular graph with parameters $((q+2)q^2,q^2-1,q^2-2,0)$. 
Thus they form a strongly regular decomposition. 
Furthermore $N_{\alpha}$, $\alpha\in\mathbb{F}_q$, commute each other by Theorem~\ref{thm:1} (ii) for $\alpha \neq \beta$.
Since each $N_{\alpha}$ has constant row and column sums in the off-diagonal blocks, each $N_{\alpha}$ commutes with $I_{q+2}\otimes (J_{q^2}-I_{q^2})$. 
Therefore the set of matrices is a commutative strongly regular decomposition. 
\end{proof}
Unfortunately, as Theorem~\ref{thm:1} (ii) showed, the matrices $N_{\alpha}$, $\alpha\in\mathbb{F}_q$,  $I_{q+2}\otimes (J_{q^2}-I_{q^2}),I_{(q+2)q^2}$ do not form a symmetric association scheme. 
Now we are going to refine the adjacency matrices in order to obtain a symmetric association scheme.

For $\alpha\in\mathbb{F}_q,i\in\{0,1,2,3\}$, we define $(0,1)$-matrices $A_{\alpha,i}$ as 
\begin{align*}
A_{\alpha,0}&=I_{q(q+2)}\otimes \phi(\alpha),\\
A_{\alpha,1}&=I_{q+2}\otimes C_{y,\alpha},\\
A_{\alpha,2}&=P_{y}\otimes C_{y,\alpha},\\
A_{\alpha,3}&=N_{\alpha}-P_{y}\otimes C_{y,\alpha}. 
\end{align*}
Note that %$I_{q+2}=P_x$, 
$A_{0,0}=I_{(q+2)q^2}$ and $\sum_{\alpha\in \mathbb{F}_q}A_{\alpha,0}=A_{0,1}$. 
Let $\mathbb{F}_q^*=\mathbb{F}_q\setminus\{0\}$.
\begin{theorem}\label{thm:as}
The set of matrices $\{A_{\alpha,0},A_{\beta,1},A_{\alpha,2},A_{\alpha,3}\mid \alpha\in\mathbb{F}_q,\beta\in\mathbb{F}_q^*\}$ is a symmetric association scheme.
\end{theorem}
\begin{proof}
By the definition of $N_{\alpha}$,  $A_{\alpha,i}$ are non-zero $(0,1)$-matrices such that $\sum_{\alpha\in\mathbb{F}_q}(A_{\alpha,0}+A_{\alpha,2}+A_{\alpha,3})+\sum_{\alpha\in\mathbb{F}_q^*}A_{\alpha,1}=J_{(q+2)q^2}$. 
Since each $P_a$ and $C_{a,\alpha}$ are symmetric, each $A_{\alpha,i}$ is symmetric. 
By $\sum_{\alpha\in \mathbb{F}_q}A_{\alpha,0}=A_{0,1}$, it is enough to show that $\mathcal{A}:=\text{span}_{\mathbb{R}}\{A_{\alpha,i}\mid \alpha\in\mathbb{F}_q,i\in\{0,1,2,3\}\}$ is closed under the matrix multiplication. 

From Lemma~\ref{lem:1}, it follows that $\mathcal{A}'=\text{span}_{\mathbb{R}}\{A_{\alpha,i}\mid \alpha\in\mathbb{F}_q,i\in\{0,1,2\}\}$ is closed under matrix multiplication. 
Since $\mathcal{A}=\mathcal{A}'+\text{span}_{\mathbb{R}}\{N_{\alpha}\mid \alpha\in\mathbb{F}_q\}$, 
we then need to calculate $A_{\alpha,i}N_{\beta},N_{\beta}A_{\alpha,i}$ and $N_{\alpha} N_{\beta}$. 
The case $N_{\alpha} N_{\beta}$ follows from Theorem~\ref{thm:1} (ii). 
For the case of $A_{\alpha,i}N_{\beta}$, we use Lemma~\ref{lem:1} to obtain 
\begin{align*}
A_{\alpha,0}N_{\beta}&=N_{\alpha+\beta},\\
A_{\alpha,1}N_{\beta}&=(J_{q+2}-I_{q+2}-P_{y})\otimes J_{q^2}+q P_{y}\otimes C_{y,\alpha+\beta}, \\ 
A_{\alpha,2}N_{\beta}&=(J_{q+2}-I_{q+2}-P_{y})\otimes J_{q^2}+q I_{q+2}\otimes C_{y,\alpha+\beta}. 
\end{align*} 
Finally $N_{\beta} A_{\alpha,i}$ follows from taking transpose of $A_{\alpha,i}N_{\beta}$ because all the adjacency matrices are symmetric. 
\end{proof}

From the proof of Theorem~\ref{thm:as}, the intersection numbers can be determined as follows. 
Precise calculations will be given in Appendix. 
Remark that the matrix $A_{0,1}$, which is not an adjacency matrix, is included in the following proposition.  

\begin{proposition}\label{prop:in}
Let $\alpha,\beta\in\mathbb{F}_q$. 
The following hold.
\begin{enumerate}
\item For $l\in\{0,3\}$, $A_{\alpha,0}A_{\beta,l}=q A_{\alpha+\beta,l}$, and for $l\in\{1,2\}$ $A_{\alpha,0}A_{\beta,l}=q A_{\beta,l}$.      
\item $A_{\alpha,1}A_{\beta,1}=A_{\alpha,2}A_{\beta,2}=q A_{\alpha+\beta,1}$ and $A_{\alpha,1}A_{\beta,2}=q A_{\alpha+\beta,2}$.  
\item $A_{\alpha,1}A_{\beta,3}=A_{\alpha,2}A_{\beta,3}=\sum_{\gamma\in\mathbb{F}_q}A_{\gamma,3}$ and $A_{\alpha,3}A_{\beta,3}=q^2A_{\alpha+\beta,0}+\sum_{\gamma\in\mathbb{F}_q}(q(-A_{\gamma,0}+A_{\gamma,1}+A_{\gamma,2})+(q-2)A_{\gamma,3})$. 
\end{enumerate}
\end{proposition}

%\begin{proposition}\label{lem:in}
%Let $\alpha,\beta\in\mathbb{F}_q$. The following hold. 
%\begin{align*}
%A_{\alpha,0}A_{\beta,0}&=A_{\alpha+\beta,0},\\
%A_{\alpha,0}A_{\beta,1}&=A_{\beta,1},\\
%A_{\alpha,0}A_{\beta,2}&=A_{\alpha+\beta,2},\\
%A_{\alpha,0}A_{\beta,3}&=A_{\alpha+\beta,3},\\
%A_{\alpha,1}A_{\beta,1}&=q A_{\alpha+\beta,1},\\
%A_{\alpha,1}A_{\beta,2}&=q A_{\alpha+\beta,2},\\
%A_{\alpha,1}A_{\beta,3}&=\sum_{\gamma\in\mathbb{F}_q} A_{\gamma,3},\\ 
%A_{\alpha,2}A_{\beta,2}&=q A_{\alpha+\beta,1},\\
%A_{\alpha,2}A_{\beta,3}&=\sum_{\gamma\in\mathbb{F}_q} A_{\gamma,3},\\ 
%A_{\alpha,3}A_{\beta,3}&=q^2A_{\alpha+\beta,0}+q \sum_{\gamma\in\mathbb{F}_q}(A_{\gamma,1}+A_{\gamma,2})-q A_{0,1}+(q-2)\sum_{\gamma\in\mathbb{F}_q} A_{\gamma,3}.
%\end{align*}
%\end{proposition}

%%%%%%%%%%%%%%%%%%%%%%%%%%%%%%%%%%%%%%%%%%%%%%%%%%%%%%%%%%%%%%%%%%%%%%%%%%%%%%%%%%%%%
\section{Eigenmatrices of symmetric association schemes and strongly regular decompositions}
We view $\mathbb{F}_{q}=\mathbb{Z}_2^m$ as the additive group. 
For $\alpha=(\alpha_1,\ldots,\alpha_m),\beta=(\beta_1,\ldots,\beta_m)\in\mathbb{Z}_2^m$, the inner product is defined by $\langle\alpha, \beta\rangle=\alpha_1\beta_1+\cdots+\alpha_m\beta_m$. 
Then for $\beta\in\mathbb{Z}_2^m$, the irreducible character denoted $\chi_{\beta}$ is $\chi_{\beta}(\alpha)=(-1)^{\langle\alpha, \beta\rangle}$ where $\alpha\in\mathbb{Z}_2^m$. 
The \defn{character table} $K$ of the abelian group $\mathbb{Z}_2^m$ is a $2^m\times 2^m$ matrix with rows and columns indexed by the elements of $\mathbb{Z}_2^m$ with $(\alpha,\beta)$-entry equal to $\chi_{\beta}(\alpha)$. 
Note that $\chi_{\beta}(\alpha)=\chi_{\alpha}(\beta)$. 
Then the Schur orthogonality relation shows $K K^\top=2^mI_{2^m}$, namely $K$ is a Hadamard matrix of order $2^m$. 
Moreover $K$ is transformed by permuting rows and columns to the Sylvester-type Hadamard matrix of order $2^m$. 
Here the Sylvester-type Hadamard matrix of order $2^m$ is defined as the $m$ tensor product of $\left(\begin{smallmatrix} 1 & 1 \\ 1 & -1 \end{smallmatrix}\right).$

To describe the primitive idempotents,  let $F_{\alpha,i}$, $\alpha\in\mathbb{F}_q,i\in\{0,1,2,3\}$, be 
\begin{align*}
F_{\alpha,i}&=\sum_{\gamma\in\mathbb{F}_q} \chi_{\alpha}(\gamma)A_{\gamma,i}.  
\end{align*}
\begin{lemma}\label{lem:f}
Let $\alpha,\beta\in\mathbb{F}_q$. 
The following hold.
\begin{enumerate}
\item For $l\in\{0,1,2,3\}$, $F_{\alpha,0}F_{\beta,l}=q\delta_{\alpha,\beta} F_{\alpha,l}$.  
\item $F_{\alpha,1}F_{\beta,1}=F_{\alpha,2}F_{\beta,2}=q^2\delta_{\alpha,\beta} F_{\alpha,1}$ and $F_{\alpha,1}F_{\beta,2}=q^2\delta_{\alpha,\beta} F_{\alpha,2}$.  
\item $F_{\alpha,1}F_{\beta,3}=F_{\alpha,2}F_{\beta,3}=q^2\delta_{\alpha,0}\delta_{\beta,0}F_{0,3}$ and $F_{\alpha,3}F_{\beta,3}=q^3 \delta_{\alpha,\beta}F_{\alpha,0}+
q^{2}\delta_{\alpha,0}\delta_{\beta,0}(q (-F_{0,0}+F_{0,1}+F_{0,2})
+(q-2)F_{0,3})$. 
\end{enumerate}
\end{lemma}
\begin{proof}
(i) for $l=0$ and (ii) for $F_{\alpha,1}F_{\beta,1}$ follow from the fact that for $i\in\{0,1\}$, the group $\{A_{\alpha,i}\mid \alpha\in\mathbb{F}_q\}$ under the matrix multiplication is isomorphic to $\mathbb{Z}_2^m$ and the Schur orthogonality. 
We prove only the case of $F_{\alpha,3}F_{\beta,3}$. 
The other cases can be proved in a similar manner.  
\begin{align*}
F_{\alpha,3}F_{\beta,3}%&=(\sum_{\gamma\in\mathbb{F}_q} \chi_{\alpha}(\gamma)A_{\gamma,3})(\sum_{\gamma'\in\mathbb{F}_q} \chi_{\beta}(\gamma')A_{\gamma',3})\\
&=\sum_{\gamma,\gamma'\in\mathbb{F}_q} \chi_{\alpha}(\gamma)\chi_{\beta}(\gamma')A_{\gamma,3}A_{\gamma',3}\\
&=\sum_{\gamma,\gamma'\in\mathbb{F}_q}\chi_{\alpha}(\gamma)\chi_{\beta}(\gamma')(q^2A_{\gamma+\gamma',0}+q \sum_{\gamma''\in\mathbb{F}_q}(-A_{\gamma'',0}+A_{\gamma'',1}+A_{\gamma'',2})+(q-2)\sum_{\gamma''\in\mathbb{F}_q} A_{\gamma'',3})\\
&=q^2 F_{\alpha,0}F_{\beta,0}+
\sum_{\gamma,\gamma'\in\mathbb{F}_q}\chi_{\alpha}(\gamma)\chi_{\beta}(\gamma')(q (-F_{0,0}+F_{0,1}+F_{0,2})+(q-2)F_{0,3})\\
&=q^3 \delta_{\alpha,\beta}F_{\alpha,0}+
q^{2}\delta_{\alpha,0}\delta_{\beta,0}(q (-F_{0,0}+F_{0,1}+F_{0,2})+(q-2)F_{0,3}). \qedhere
\end{align*}
\end{proof}

Let $E_{i},E_{\alpha,j}$ be 
\begin{align*}
E_0&=\frac{1}{(q+2)q^2}(F_{0,1}+F_{0,2}+F_{0,3}),\\
E_1&=\frac{1}{(q+2)q^2}(\frac{q}{2}(F_{0,1}+F_{0,2})-F_{0,3}),\\
E_{\alpha,1}&=\frac{1}{2q^2}(F_{\alpha,1}+F_{\alpha,2}), \quad \alpha\in\mathbb{F}_q,\\
E_{\alpha,2}&=\frac{1}{2q^2}(F_{\alpha,1}-F_{\alpha,2}), \quad \alpha\in\mathbb{F}_q,\\
E_{\alpha,3}&=\frac{1}{2q^2}(q F_{\alpha,0}+F_{\alpha,3}), \quad \alpha\in\mathbb{F}_q,\\
E_{\alpha,4}&=\frac{1}{2q^2}(q F_{\alpha,0}-F_{\alpha,3}), \quad \alpha\in\mathbb{F}_q.
\end{align*}
Note that $E_{0,1}=E_0+E_1$, $E_{0,3}=(\frac{q}{2}+1)E_{0}+\sum_{\gamma\in\mathbb{F}_q^*}E_{\gamma,1}$ and $E_{0,4}=(-\frac{q}{2}+1)E_0+2E_1+\sum_{\gamma\in\mathbb{F}_q^*}E_{\gamma,1}$. 
From Lemma~\ref{lem:f}, the following is readily obtained. 
\begin{theorem}\label{thm:pi}
The matrices $E_0,E_1,E_{\beta,1},E_{\alpha,2},E_{\beta,3},E_{\beta,4}$, $\alpha\in\mathbb{F}_q,\beta\in\mathbb{F}_q^*$, are the primitive idempotents of the symmetric association scheme. 
\end{theorem}

As a direct consequence of Theorem~\ref{thm:pi} and Lemma~\ref{lem:eigen}, the eigenmatrices of the symmetric association scheme are determined as follows. 
\begin{theorem}\label{thm:eigen}
The first and second eigenmatrices $P,Q$  of the symmetric association scheme are given as follows:
\begin{align*}
P=\bordermatrix{
                   & A_{\alpha,0}              & A_{\beta,1}                  & A_{\alpha,2} & A_{\alpha,3} \cr
E_0              & 1                            & q                               & q & q^2 \cr
E_1              & 1                            & q                               & q & -2q \cr
E_{\beta',1}    & 1                            & q\chi_{\beta'}(\beta)     & q\chi_{\beta'}(\alpha)  & 0  \cr
E_{\alpha',2}   & 1                            & q\chi_{\alpha'}(\beta)   & -q\chi_{\alpha'}(\alpha) & 0 \cr
E_{\beta',3}    & \chi_{\beta'}(\alpha)    & 0                             & 0  &  q\chi_{\beta'}(\alpha) \cr
E_{\beta',4}    & \chi_{\beta'}(\alpha)    & 0                             & 0  & -q\chi_{\beta'}(\alpha) \cr
},
\end{align*}
\begin{align*}
Q=\bordermatrix{
              & E_0 & E_1 & E_{\beta',1}                          & E_{\alpha',2}                       & E_{\beta',3} & E_{\beta',4} \cr
A_{\alpha,0}          & 1    & \frac{q}{2} & \frac{q}{2}+1                        & \frac{q}{2}+1                          & (\frac{q}{2}+1)q\chi_{\beta'}(\alpha) & (\frac{q}{2}+1)q\chi_{\beta'}(\alpha) \cr
A_{\beta,1}  & 1    & \frac{q}{2} & (\frac{q}{2}+1)\chi_{\beta'}(\beta) & (\frac{q}{2}+1)\chi_{\alpha'}(\beta)   & 0 & 0  \cr
A_{\alpha,2} & 1   & \frac{q}{2} &(\frac{q}{2}+1)\chi_{\beta'}(\alpha)  & -(\frac{q}{2}+1)\chi_{\alpha'}(\alpha)  & 0 & 0 \cr
A_{\alpha,3} & 1   & -1   & 0                           & 0 & (\frac{q}{2}+1)\chi_{\beta'}(\alpha) & -(\frac{q}{2}+1)\chi_{\beta'}(\alpha) \cr
},
\end{align*}
where $\alpha,\alpha'\in\mathbb{F}_q,\beta,\beta'\in\mathbb{F}_q^*$. 
\end{theorem}
\begin{proof}
The formula for the second eigenmatrix $Q$ is derived from Theorem~\ref{thm:pi}. 
Since the valencies of $A_{\alpha,0},A_{\beta,1},A_{\alpha,2},A_{\alpha,3}$ are respectively $1,q,q,q^2$,  
the formula for the first eigenmatrix $P$ follows from Lemma~\ref{lem:eigen}. 
\end{proof}
As a corollary of the formula for the first eigenmatrix $P$, we obtain several commuting strongly regular graphs.
In order to describe the eigenamtrix, we need the notion of the permutation automorphism group of Sylvester-type Hadamard matrices. 
Define the \defn{permutation automorphism group} of the Sylvester-type Hadamard matrix $K=(\chi_{\beta}(\alpha))_{\alpha,\beta\in\mathbb{F}_q}$ by 
\begin{align*}
\mathrm{PAut}(K)=\{(\sigma,\tau)\in S(\mathbb{F}_q)\times S(\mathbb{F}_q) \mid \chi_{\tau(\beta)}(\sigma(\alpha))=\chi_{\beta}(\alpha) \text{ for any }\alpha,\beta\in\mathbb{F}_q\}, 
\end{align*}
where $S(\mathbb{F}_q)$ denotes the set of all permutations on the set $\mathbb{F}_q$.  
Note that for $(\sigma,\tau)\in\mathrm{PAut}(K)$, $\sigma(0)=\tau(0)=0$. 

\begin{corollary}
For any $\sigma\in S(\mathbb{F}_q)$, 
the set of matrices $\{A_{0,0}, \sum_{\alpha\in\mathbb{F}_q^*}(A_{\alpha,0}+A_{\alpha,1}), A_{\alpha,2}+A_{\sigma(\alpha),3} \mid \alpha\in\mathbb{F}_q\}$ is a commutative strongly regular decomposition. 

If there exists $\tau\in S(\mathbb{F}_q)$ such that $(\sigma,\tau)\in\text{PAut}(K)$, 
then the eigenmatrix is given by 
\begin{align*}
P=\bordermatrix{
              & A_{0,0}  & \sum_{\alpha\in\mathbb{F}_q^*}(A_{\alpha,0}+A_{\alpha,1})                                                & A_{\alpha,2}+A_{\sigma(\alpha),3} \cr
E_0         & 1     & q^2-1                                                 & q^2+q \cr
E_1+E_{0,2}         & 1     & q^2-1                                                 & -q \cr
E_{\beta',1}+E_{\tau(\beta'),3}          & 1     &  -1                       & q\chi_{\beta'}(\alpha)   \cr
E_{\beta',2}+E_{\tau(\beta'),4}     & 1     &  -1                            & -q\chi_{\beta'}(\alpha) \cr
}.   
\end{align*}

If there does not exist $\tau\in S(\mathbb{F}_q)$ such that $(\sigma,\tau)\in \text{PAut}(K)$, 
then the eigenmatrix is given by 
\begin{align*}
P=\bordermatrix{
              & A_{0,0}  & \sum_{\alpha\in\mathbb{F}_q^*}(A_{\alpha,0}+A_{\alpha,1})                                                & A_{\alpha,2}+A_{\sigma(\alpha),3} \cr
E_0         & 1     & q^2-1                                                 & q^2+q \cr
E_1+E_{0,2}         & 1     & q^2-1                                                 & -q \cr
E_{\beta',1}         & 1     &  -1                            & q\chi_{\beta'}(\alpha)   \cr
E_{\beta',2}     & 1     &  -1                            & -q\chi_{\beta'}(\alpha) \cr
E_{\beta',3}  & 1     & -1                            &   q\chi_{\beta'}(\sigma(\alpha)) \cr
E_{\beta',4} & 1    & -1                             &  -q\chi_{\beta'}(\sigma(\alpha)) \cr
}.
\end{align*}
In both cases above, $\alpha,\alpha'\in\mathbb{F}_q$ and $\beta,\beta'\in\mathbb{F}_q^*$.
\end{corollary}

%%%%%%%%%%%%%%%%%%%%%%%%%%%%%%%%%%%%%%%%%%%%%%%%%%%%%%%%%%%%%%%%%%%%%%%%%%%%%%%%%%%%%%%%%
%%%%%%%%%%%%%%%%%%%%%%%%%%%%%%%%%%%%%%%%%%%%%%%%%%%%%%%%%%%%%%%%%%%%%%%%%%%%%%%%%%%%%%%%%
\noindent {\bf Acknowledgments.}
Hadi Kharaghani is supported by an NSERC Discovery Grant.  Sho Suda is supported by JSPS KAKENHI Grant Number 15K21075.

%%%%%%%%%%%%%%%%%%%%%%%%%%%%%%%%%%%%%%%%%%%%%

\appendix
\def\thesection{\Alph{section}}

\section{Intersection numbers} \label{appendix}
We calculate the intersection numbers of the symmetric association schemes. Let $\alpha,\beta\in\mathbb{F}_q$.  

Proposition~\ref{prop:in}(i): 
\begin{align*}
A_{\alpha,0}A_{\beta,0}&=I_{q(q+2)}\otimes \phi(\alpha+\beta)=A_{\alpha+\beta,0}. \\
A_{\alpha,0}A_{\beta,1}&=I_{q+2}\otimes ((I_q\otimes \phi(\alpha))C_{y,\beta})=I_{q+2}\otimes C_{y,\beta}=A_{\beta,1}.\\
A_{\alpha,0}A_{\beta,2}&=P_{y}\otimes ((I_q\otimes \phi(\alpha))C_{y,\beta})=P_{y}\otimes C_{y,\beta}=A_{\beta,2}.\\
A_{\alpha,0}A_{\beta,3}&=(I_{q+2}\otimes (I_q\otimes \phi(\alpha))(N_{\beta}-P_{y}\otimes C_{y,\beta})\\
&=(I_{q+2}\otimes (I_q\otimes \phi(\alpha))((C_{L(a,a'),\beta})_{a,a'\in S}-P_{y}\otimes C_{y,\beta})\\
&=((C_{L(a,a'),\alpha+\beta})_{a,a'\in S}-P_{y}\otimes C_{y,\alpha+\beta}\\
&=N_{\alpha+\beta}-P_{y}\otimes C_{y,\alpha+\beta}\\
&=A_{\alpha+\beta,3}.
\end{align*}

Proposition~\ref{prop:in}(ii):
\begin{align*}
A_{\alpha,2}A_{\beta,2}&=(I_{q+2}\otimes C_{y,\alpha})(I_{q+2}\otimes C_{y,\beta})=q I_{q+2}\otimes C_{y,\alpha+\beta}=q A_{\alpha+\beta,1}.  \\
A_{\alpha,1}A_{\beta,2}&=q I_{q+2}\otimes C_{y,\alpha+\beta}=q A_{\alpha+\beta,2}. \\
A_{\alpha,2}A_{\beta,2}&=(P_{y}\otimes C_{y,\alpha})(P_{y}\otimes C_{y,\beta})=q I_{q+2}\otimes C_{y,\alpha+\beta}=q A_{\alpha+\beta,1}. 
\end{align*}

Proposition~\ref{prop:in}(iii): 
\begin{align*}
A_{\alpha,1}A_{\beta,3}&=(I_{q+2}\otimes C_{y,\alpha})(N_{\beta}-P_{y}\otimes C_{y,\beta})\\
&=(I_{q+2}\otimes C_{y,\alpha})N_{\beta}-P_{y}\otimes (C_{y,\alpha}C_{y,\beta})\\
&=((J_{q+2}-I_{q+2}-P_{y})\otimes J_{q^2}+qI_{q+2}\otimes C_{y,\alpha+\beta})-q P_{y}\otimes C_{y,\alpha+\beta}\\
&=(J_{q+2}-I_{q+2}-P_{y})\otimes J_{q^2}\\
&=\sum_{\gamma\in\mathbb{F}_q} A_{\gamma,3}. \\
A_{\alpha,2}A_{\beta,3}&=(P_{y}\otimes C_{y,\alpha})(N_{\beta}-P_{y}\otimes C_{y,\beta})\\
&=(P_{y}\otimes C_{y,\alpha})N_{\beta}-(P_{y}\otimes C_{y,\alpha})\otimes (P_{y}\otimes C_{y,\beta})\\
&=(J_{q+2}-I_{q+2}-P_{y})\otimes J_{q^2}+qI_{q+2}\otimes C_{y,\alpha+\beta})-qI_{q+2}\otimes C_{y,\alpha+\beta}\\
&=(J_{q+2}-I_{q+2}-P_{y})\otimes J_{q^2}\\
&=\sum_{\gamma\in\mathbb{F}_q} A_{\gamma,3}. 
\end{align*}

We use the following formula for the permutation matrices. 
\begin{align*}
\sum_{\alpha,\beta\in\mathbb{F}_q,\alpha\neq \beta} P_{\alpha}P_{\beta}&=\sum_{\alpha\in\mathbb{F}_q} P_{\alpha}(\sum_{\beta\in\mathbb{F}_q,\beta\neq \alpha}P_{\beta})\\
&=\sum_{\alpha\in\mathbb{F}_q} P_{\alpha}(J_{q+2}-I_{q+2}-P_{\alpha}-P_{y})\\
&=\sum_{\alpha\in\mathbb{F}_q} (J_{q+2}-P_{\alpha}-I_{q+2}-P_{\alpha}P_{y})\\
&=q(J_{q+2}-I_{q+2})-(\sum_{\alpha\in\mathbb{F}_q}P_{\alpha})(I_{q+2}+P_{y})\\
&=q(J_{q+2}-I_{q+2})-(J_{q+2}-I_{q+2}-P_{y})(I_{q+2}+P_{y})\\
&=(q-2)(J_{q+2}-I_{q+2})+2P_{y}.  
\end{align*}This yields
\begin{align*}
A_{\alpha,3}A_{\beta,3}&=(\sum_{\gamma\in\mathbb{F}_q} P_{\gamma} \otimes C_{\gamma,\alpha})(\sum_{\gamma'\in\mathbb{F}_q} P_{\gamma'}\otimes C_{\gamma',\beta})\\
&=\sum_{\gamma,\gamma'\in\mathbb{F}_q} P_{\gamma}P_{\gamma'} \otimes C_{\gamma,\alpha}C_{\gamma',\beta}\\
&=q\sum_{\gamma\in\mathbb{F}_q} I_{q+2} \otimes C_{\gamma,\alpha+\beta}+\sum_{\gamma,\gamma\in\mathbb{F}_q,\gamma\neq \gamma'} P_{\gamma}P_{\gamma'} \otimes J_{q^2}\\
&=qI_{q+2} \otimes (qI_q\otimes \phi(\alpha+\beta)+(J_q-I_q)\otimes J_q
)+((q-2)(J_{q+2}-I_{q+2})+2P_{y}) \otimes J_{q^2}\\
&=q^2I_{q(q+2)} \otimes \phi(\alpha+\beta)+qI_{q+2}\otimes(J_q-I_q)\otimes J_q
+((q-2)(J_{q+2}-I_{q+2})\otimes J_{q^2}+2P_{y}\otimes J_{q^2}\\
&=q^2A_{\alpha+\beta,0}+\sum_{\gamma\in\mathbb{F}_q}(q(-A_{\gamma,0}+A_{\gamma,1}+A_{\gamma,2})+(q-2)A_{\gamma,3}).
\end{align*}

\section{Krein numbers} \label{appendix:b}
We calculate the entrywise product of the primitive idempotents. 
From the equations below, we may find the Krein numbers for the symmetric association scheme by making use of  
$
E_{0,1}=E_0+E_1, 
E_{0,3}=(\frac{q}{2}+1)E_{0}+\sum_{\gamma\in\mathbb{F}_q^*}E_{\gamma,1}, 
E_{0,4}=(-\frac{q}{2}+1)E_0+2E_1+\sum_{\gamma\in\mathbb{F}_q^*}E_{\gamma,1} 
$. 
For $\alpha,\alpha'\in\mathbb{F}_q,\beta,\beta'\in\mathbb{F}_q^*$, the following are readily obtained by Theorem~\ref{thm:eigen}:
\begin{align*}
E_1\circ E_1&=\frac{1}{(q+2)q^2}(\frac{q}{2}E_0+\frac{q-2}{2}E_1),\\
E_1\circ E_{\beta,1}&=\frac{1}{(q+2)q^2}\frac{q}{2}E_{\beta,1},\\
E_1\circ E_{\alpha,2}&=\frac{1}{(q+2)q^2}\frac{q}{2}E_{\alpha,2},\\ \displaybreak[0]
E_1\circ E_{\beta,3}&=\frac{1}{(q+2)q^2}(\frac{q+2}{4}E_{\beta,3}+\frac{q-2}{4}E_{\beta,4}),\\ \displaybreak[0]
E_1\circ E_{\beta,4}&=\frac{1}{(q+2)q^2}(\frac{q-2}{4}E_{\beta,3}+\frac{q+2}{4}E_{\beta,4}),\\ \displaybreak[0]
E_{\beta,1}\circ E_{\beta',1}&=\frac{1}{(q+2)q^2}\frac{q+2}{2}E_{\beta+\beta',1},\\ \displaybreak[0]
E_{\beta,1}\circ E_{\alpha,2}&=\frac{1}{(q+2)q^2}\frac{q+2}{2}E_{\alpha+\beta,2},\\ \displaybreak[0]
E_{\alpha,2}\circ E_{\alpha',2}&=\frac{1}{(q+2)q^2}\frac{q+2}{2}E_{\alpha+\alpha',1},\\ \displaybreak[0]
E_{\beta,3}\circ E_{\beta',3}&=E_{\beta,4}\circ E_{\beta',4}=\frac{1}{(q+2)q^2}(\frac{q}{2}+1)(\frac{q+1}{2}E_{\beta+\beta',3}+\frac{q-1}{2}E_{\beta+\beta',4}),\\ \displaybreak[0]
E_{\beta,3}\circ E_{\beta',4}&=\frac{1}{(q+2)q^2}(\frac{q}{2}+1)(\frac{q-1}{2}E_{\beta+\beta',3}+\frac{q+1}{2}E_{\beta+\beta',4}),\\ \displaybreak[0]
E_{\beta,1}\circ E_{\beta',3}&=E_{\beta,1}\circ E_{\beta',4}=\frac{1}{(q+2)q^2}(\frac{q+2}{4}E_{\beta',3}+\frac{q+2}{4}E_{\beta',4}),\\ \displaybreak[0]
E_{\alpha,2}\circ E_{\beta,3}&=E_{\alpha,2}\circ E_{\beta,4}=\frac{1}{(q+2)q^2}(\frac{q+2}{4}E_{\beta,3}+\frac{q+2}{4}E_{\beta,4}).
\end{align*}


\begin{thebibliography}{99}
\bibitem{BI}
E. Bannai and T. Ito, 
{\sl Algebraic Combinatorics I: Association Schemes},
{Benjamin/Cummings, Menlo Park, CA,} 1984.


\bibitem{D}
E. van Dam, 
Strongly regular decompositions of the complete graph, 
{\sl J.\ Algebraic Combin.} {\bf 17} (2003), 181--201. 

\bibitem{K}
H. Kharaghani, 
New class of weighing matrices, 
{\sl Ars.\ Combin.} {\bf 19} (1985), 69-72. 

\bibitem{S}
D. R. Stinson, 
Combinatorial Designs: Constructions and Analysis, New York, Springer, 2004. 



\end{thebibliography}
\end{document}